\documentclass[12pt]{amsart}

\def\rar{\rightarrow}

\usepackage{color}
\definecolor{red}{rgb}{1.00,0.00,0.00}

\newtheorem{theorem}{Theorem}[section]
\newtheorem{lemma}[theorem]{Lemma}
\newtheorem{corollary}[theorem]{Corollary}
\newtheorem{proposition}[theorem]{Proposition}
\newtheorem{example}[theorem]{Example}
\newtheorem*{remark}{Remark}
\newtheorem{definition}[theorem]{Definition}

\newtheorem{question}{Question}
\newtheorem*{notations}{Notations}

\def\sA{\langle A\rangle}
\def\sB{\langle B\rangle}
\def\sC{\langle C\rangle}
\def\sD{\langle D\rangle}
\def\sS{\langle S\rangle}

\def\rA{k[A]}
\def\rB{k[B]}
\def\rC{k[C]}
 \def\rS{k[S]}

\def\N{\mathbb{N}}

\def\Z{\mathbb{Z}}

\def\a{{\bf a}}
\def\b{{\bf b}}
\def\iff{if and only if}
\def\rk#1{\hbox{\rm rank\,(#1)}}
\newcommand{\pd}{\mathrm {pd }}
\newcommand{\depth}{\mathrm {depth }}

\title[Gluing semigroups]{Gluing semigroups - when and how}

\author{Philippe Gimenez}
 \address{Instituto de Matem\'aticas de la Universidad de Valladolid (IMUVA),
Facultad de Ciencias, 47011 Valladolid, Spain.}
 \email{pgimenez@agt.uva.es}
 \thanks{The first author was partially supported by 
{\it Ministerio de Ciencia e Innovaci\'on} (Spain) MTM2016-78881-P and 
{\it Consejer\'{\i}a de Educaci\'on de la Junta de Castilla y Le\'on} VA128G18.}

\author{Hema Srinivasan}
 \address{Mathematics Department, University of
Missouri, Columbia, MO 65211, USA.}
 \email{SrinivasanH@math.missouri.edu}
\thanks{The second author acknowleges with pleasure the support and hospitality of University of Valladolid while working on this project.}

\thanks{{\bf Keywords}: semigroup rings, gluing, degenerate semigroups, Cohen-Macaulay rings. }

\begin{document}
\maketitle

\begin{abstract} 
Given two semigroups $\sA$ and $\sB$ in $\N^n$, we wonder when they can be glued, i.e., when there exists a semigroup $\sC$ in $\N^n$ such that the defining ideals of the
corresponding semigroup rings satisfy that $I_C=I_A+I_B+\langle\rho\rangle$ for some binomial $\rho$. 
If $n\geq 2$ and  $\rA$ and $\rB$ are Cohen-Macaulay,  we prove that in order to glue them, one of  the two semigroups must be degenerate.  Then we study the two most degenerate cases: when one of the semigroups is generated by one single element (simple split) and the case
where it is generated by at least two elements and all the elements of the semigroup lie on a line. 
In both cases we characterize the semigroups that can be glued and say how to glue them.   Further,  in these cases, we conclude that the glued $\sC$ is Cohen-Macaulay if and only if both $\sA$ and $\sB$ are also Cohen-Macaulay.  
As an application,  we characterize precisely the Cohen-Macaulay semigroups that can be glued when $n=2$.
\end{abstract}

\section{Introduction}

Let $\sA$ be the semigroup finitely generated by a subset $A=\{\a_1,\ldots,\a_p\}$ of $\N^n$ and $k$ an arbitrary field.
If $\phi_A: k[x_1, \dots, x_p] \to k[t_1,\ldots, t_n]$ is the ring homomorphism given by 
$\phi_A(x_j) = t^{\a_j}=\prod_{i=1}^{n}t_i^{a_{ij}}$ where $\a_j=\begin{pmatrix} a_{1j}\\ \vdots\\ a_{nj}\end{pmatrix}\in\N^n$,
the kernel of $\phi_A$, $I_A=\ker(\phi_A)$, is a binomial prime ideal and the semigroup ring $\rA$ is isomorphic to $k[x_1, \dots, x_p]/I_A$.
We will also denote by $A$ the $n\times p$ integer matrix whose columns are the elements in $A$.

\medskip
Inspired by the classical construction by Delorme in \cite{De} for the study and characterization
of complete intersection numerical semigroups, Rosales introduced  in \cite{rosales} the concept of gluing. For a semigroup $\sC$, 
when the set of generators of the semigroup splits into two disjoint parts, $C=A\cup B$, such that 
by $I_C=I_A+ I_B +\langle\rho\rangle$ where $\rho$ is a binomial whose first, respectively second, monomial involves only variables corresponding to elements
in $A$, respectively $B$, we say that $\sC$ is a {\it gluing} of $\sA$ and $\sB$.
When this occurs, we also say that
the semigroup $\sC$ is {\it decomposable} or that it  {\it splits} (or {\it decomposes}) as $\sC=\sA\sqcup \sB$.
This property can be characterized in terms of the semigroups $\sA$ and $\sB$ and the subgroups in $\Z^n$ associated to them; see
\cite[Thm. 1.4]{rosales}.

\medskip
Let's fix some notations that we will use along the paper.
If we have two semigroups $\sA$ and $\sB$ in $\N^n$ with  $A=\{\a_1,\ldots,\a_p\}$ and $B=\{\b_1,\ldots,\b_q\}$, variables corresponding 
to $A$, respectively $B$, will be denoted by $x_1,\ldots,x_p$, respectively $y_1,\ldots,y_q$.
Thus, $I_A\subset k[x_1,\ldots,x_p]$, $\rA\simeq k[x_1,\ldots,x_p]/I_A$,
$I_B\subset k[y_1,\ldots,y_q]$ and $\rB\simeq k[y_1,\ldots,y_q]/I_B$.
If the generating set $C$ of a semigroup $\sC$ splits into two disjoint parts $C=A\cup B$,
then $I_C\subset R=k[x_1,\ldots,x_p,y_1,\ldots,y_q]$ and $\rC\simeq R/I_C$. 
Since multiplying by a common integer all the elements in the generating set of a semigroup does not change the defining ideal of the semigroup ring,
one can easily check that if $C=k_1A\cup k_2B$ for some nonnegative integers $k_1$ and $k_2$, then $I_C\cap k[x_1,\ldots,x_p]=I_A(=I_{k_1A})$ and $I_C\cap k[y_1,\ldots,y_q]=I_B(=I_{k_2B})$.
Note that if one gives weight $k_1\a_i$ to $x_i$ and
$k_2\b_j$ to $y_j$ for all $i,j$, then the ring $\rC$ is graded over the semigroup $\sC$.

\medskip
In this paper, we are interested in studying when two semigroups $\sA$ and $\sB$ in $\N^n$ can be glued in the following sense:

\begin{definition}{\rm
Given an interger $n\geq 1$ and two subsets $A=\{\a_1,\ldots,\a_p\}$ and $B=\{\b_1,\ldots,\b_q\}$ in $\N^n$, we say that the
semigroups $\sA$ and $\sB$ {\it can be glued} if there exist two integers $k_1,k_2\in\N$ such that for
$C=k_1A\cup k_2B$, the semigroup $\sC$ is a gluing of $\langle k_1 A\rangle$ and $\langle k_2 B\rangle$, i.e., 
$I_C=I_A+ I_B +\langle\rho\rangle$ for some binomial $\rho=\underbar{x}^\alpha-\underbar{y}^\beta$ with $\alpha\in\N^p$ and $\beta\in\N^q$.
When this occurs, we will say that $\sC$ is a gluing of $\sA$ and $\sB$ instead of saying that it is a gluing of 
$\langle k_1 A\rangle$ and $\langle k_2 B\rangle$.
}\end{definition}

\begin{remark}{\rm
In the definition of gluing, one can always assume that $k_1$ and $k_2$ are relatively prime, if needed.    
}\end{remark}

We state the following problems:

\begin{question}[When and how]\label{mainquest}{\rm
Given two semigroups $\sA$ and $\sB$ in $\N^n$ can $\sA$ and $\sB$ be glued?  When it is possible to glue them, what should the integers $k_1$ and $k_2$ be 
so that for $C=k_1A\cup k_2B$, $\sC$ is a gluing of $\sA$ and $\sB$.}
\end{question}

\begin{question}[The Cohen-Macaulay property]\label{weakerquest}{\rm
Given two semigroups $\sA$ and $\sB$ in $\N^n$, such that $\rA$ and $\rB$ are Cohen-Macaulay,   can $\sA$ and $\sB$ be glued?  When the answer is positive, is the resulting glued semigroup ring $\rC$ also Cohen-Macaulay? } \end{question}

The case of numerical semigroups is well understood.
Recall that if $n=1$, every semigroup ring is Cohen-Macaulay.
Moreover, it is well known that given two arbitrary numerical semigroups $\sA$ and $\sB$,
if one chooses $k_1\in\sB$ and $k_2\in\sA$,
then for $C=k_1A\cup k_2B$, one has that $I_C=I_A+I_B+\langle\rho\rangle$ for some binomial $\rho=\underbar{x}^\alpha-\underbar{y}^\beta$
with $\alpha\in\N^p$ and $\beta\in\N^q$. 
One can thus answer to the above questions when $n=1$: two numerical semigroups can always be glued and one knows how to glue them (choosing
$k_1\in\sB$ and $k_2\in\sA$). Moreover, if $\sC$ is a gluing of $\sA$ and $\sB$,
the semigroup rings $\rA$, $\rB$ and $\rC$ are always Cohen-Macaulay in this case.
There exists no similar construction when $n\geq 2$. 

\medskip
In section \ref{secProjDim}, we partially answer to question \ref{weakerquest} and  show that if the rings $\rA$ and $\rB$ have dimension $n$,
i.e., $\sA$ and $\sB$ are nondegenerate, and both rings 
are Cohen-Macaulay, then the semigroups $\sA$ and $\sB$ can not be glued if $n\geq 2$ (theorem \ref{thmNondegen}). Degeneracy is thus necessary
in order to glue Cohen-Macaulay semigroups when $n\geq 2$. 
In section \ref{secSimplesplit}, we focus on the case of a simple split, i.e., when $B$ has only one element.
By definition, if $q=1$ then $\sB$ is degenerate whenever $n\geq 2$ and we  give complete answers to both questions in this case (theorems \ref{thmSimplesplit}
and \ref{thmGlue1elt} and corollary \ref{corCMchar1elt}).
In section \ref{secLine}, we consider another degenerate case: when all the generators of $\sB$ lie on a line, i.e.,
when $q\geq 2$ and the matrix $B$ has rank 1 (theorem \ref{thmLineChar} and corollary \ref{corCMcharLine}).
Putting all together, we then give a complete answer to question \ref{weakerquest} when $n=2$ in section \ref{secConsequences}.

\section{Degeneracy}\label{secProjDim}

It is well-known that given $A=\{\a_1,\ldots,\a_p\}$ in $\N^n$, the Krull dimension of the semigroup ring $\rA$ coincides with the rank of the $n\times p$
integer matrix $A$ whose columns are $\a_1,\ldots,\a_p$; see, e.g., \cite[Lem. 4.2]{sturm}. 
In particular, $\dim\rA\leq n$ and we will say that $\sA$ is {\it nondegenerate} if the dimension of $\rA$ is $n$.
Note that if we don't have enough generators, i.e., if $p<n$, then $\sA$ is always degenerate, and if $p\geq n$, $\sA$ is nondegenerate \iff{}
the matrix $A$ has maximal rank.

\begin{remark}\label{maximalrank} {\rm
If the $n\times p$ matrix $A$ is not of rank $n$ then, reordering eventually the rows of the matrix,
there are rational numbers $r_i$ such that $a_{nj} = \sum_{i=1}^{n-1} r_ia_{ij}$.  Hence, there is a positive integer $d$ such that, if $A'$ is the $n-1\times p$ matrix of the first $n-1$ rows of $dA$, then the semigroup rings $k[A']$ and $\rA$ are isomorphic.  
}\end{remark}

In \cite{jpaa19}, we discribed the minimal graded free resolution of $\rC$ in terms of those of $\rA$ and $\rB$ when
$\sC$ is a gluing of $\sA$ and $\sB$. Let's recall here our main result:

\begin{theorem}[{\cite[Thm. 6.1, Cor. 6.2]{jpaa19}}]\label{thmJpaa}
Let $A=\{\a_1,\ldots,\a_p\}$ and $B=\{\b_1,\ldots,\b_q\}$ be two finite subsets of $\N^n$ and assume that
$\sC$ is a gluing of $\sA$ and $\sB$, i.e., $C=A\cup B$ and $I_C=I_A+I_B+\langle \rho\rangle$ for some $\rho=\underbar{x}^\alpha-\underbar{y}^\beta$ 
with $\alpha\in\N^p$ and $\beta\in\N^q$. Consider $F_A$ and $F_B$, minimal graded free resolutions of $\rA$ and $\rB$.
\begin{enumerate}
\item
A minimal graded free resolution of $\rC$ can be obtained as the mapping cone of $\rho:F_A\otimes F_B\rar F_A\otimes F_B$ where
$\rho$ is induced by multiplication by $\rho$.
\item\label{betti}
The Betti numbers of $\rA$, $\rB$ and $\rC$ are related as follows: $\forall i\geq 0$,
$$
\beta_i(\rC) =
\sum _{i'=0}^i \beta_{i'}(\rA)[\beta_{i-i'}(\rB)+\beta_{i-i'-1}(\rB)] =
\sum _{i'=0}^i \beta_{i'}(\rB)[\beta_{i-i'}(\rA)+\beta_{i-i'-1}(\rA)]\,.
$$
\item\label{projdim}
The relation between the projective dimensions of $\rA$, $\rB$ and $\rC$ is:
$$
\pd(\rC)=\pd(\rA)+\pd(\rB)+1\,.
$$
\end{enumerate}
\end{theorem}

Using the last part of the previous result, one can easily show that the only nondegenerate semigroups whose semigroup ring is Cohen-Macaulay that
can be glued are the numerical semigroups.

\begin{theorem}\label{thmNondegen}
Let $\sA$ and $\sB$ be two nondegenerate semigroups in $\N^n$ such that $\rA$ and $\rB$ are Cohen-Macaulay.
Then $\sA$ and $\sB$ can be glued if and only if $n=1$.
\end{theorem}

\begin{proof}
As we already mentioned in the introduction, if $n=1$, then two arbitrary semigroups $\sA$ and $\sB$ can always be glued since for any $k_1\in\sB$ and
any $k_2\in\sA$, for $C=k_1A\cup k_2B$, one has that $\sC$ is a gluing of $\sA$ and $\sB$. Moreover, in this case all the semigroup rings
$\rA$, $\rB$ and $\rC$ are Cohen-Macaulay.

\smallskip
Conversely, assume that $n\geq 2$. Take two finite subsets $A$ and $B$ in $\N^n$ with $p$ and $q$ elements respectivelly, and assume
that $\sA$ and $\sB$ are nondegenerate, i.e., $\dim\rA=\dim\rB=n$, and that both rings $\rA$ and $\rB$ are Cohen-Macaulay.
Then, by the Auchslander-Bushbaum formula, of has that
$\pd (\rA)=p-n$ and $\pd (\rB)=q-n$. If $\sA$ and $\sB$ could be glued, then by theorem~\ref{thmJpaa}~(\ref{projdim}), one would have that 
$\pd (\rC)=p+q-2n+1$.
On the other hand, if $\dim\rA=n$ then $\dim\rC=n$ so, again by the Auchslander-Bushbaum formula, the projective
dimension of $\rC$ should be at least $p+q-n$. A contradiction.
\end{proof}

For $n\geq 2$, degeneracy is hence necessary in order to glue two Cohen-Macaulay semigroups.

\begin{example}{\rm
If $\sS\subset\N^2$ is the semigroup generated by $S=\{(3,0),(2,1),(1,2),(0,3)\}$,
the ideal $I_S$ is the defining ideal of the twisted cubic which is known to be Cohen-Macaulay.
By theorem \ref{thmNondegen}, $\sS$ can not be glued with itself in $\N^2$.
But one can consider the two degenerate semigroups $\sA$ and $\sB$ of $\N^3$ generated
respectively by $A=\{(4, 0, 0), (3, 1, 0), (2, 2, 0), (1, 3, 0)\}$ and 
$B=\{(3, 3, 0), (3, 2, 1),(3, 1, 2), (3, 0, 3)\}$, whose defining ideals $I_A\subset k[x_1,\ldots,x_4]$ and $I_B\subset k[y_1,\ldots,y_4]$ 
are both the defining ideal of the twisted cubic. In other words, $\rA\simeq \rB\simeq \rS$ and $\sA$ and $\sB$ can be thought as
two copies of $\sS$ in $\N^3$ where they are degenerate.
As shown in \cite[Ex. 7]{jpaa19}, $\sA$ and $\sB$ can be glued because if $C=A\cup B$,
then $I_C=I_A+I_B+\langle y_1^2-x_1x_4^2\rangle$.
}\end{example}

We will focus now on the case where one of the two semigroups, for example $\sB$, is the most degenerate possible, that is
when the dimension of $\rB$ is 1. This happens when $q=1$ (simple split) or when $q\geq 2$ and the matrix $B$ has rank 1,
i.e., when there exists $\b\in\N^n$ such that $B = \b \begin{bmatrix}u_1 & \ldots & u_q\end{bmatrix}$ for $u_1, \ldots, u_q \in \N$
(the elements lie on a line).
We focus on those two cases in sections \ref{secSimplesplit} and \ref{secLine} respectivelly.

\section{Simple split}\label{secSimplesplit}

Assume in this section that $q=1$. The only element in $B$ and the corresponding variable will be denoted here by $\b$ and $y$ 
(instead of $\b_1$ and $y_1$ respectively). So $B=\{\b\}$.
In this section, we will not assume that $\rA$ is Cohen-Macaulay since the results that we state here are valid with no hypothesis on $\sA$.

\medskip
One has that $I_B=(0)$ and $\rB\simeq k[y]$. So given two relatively prime integers $k_1$ and $k_2$, for 
$C=k_1 A\cup \{k_2\b\}$, one has that $\sC$ is a gluing \iff{} $I_C=I_A+\langle \rho\rangle$ 
for some binomial $\rho$ of the form $ \rho=y^d - x_1^{\alpha_1}\cdots x_p^{\alpha_p}$ being $d,\alpha_1,\ldots,\alpha_p$
positive integers  and $d\neq 0$.
This implies that 
\begin{equation}\label{eq1}
dk_2\b=\alpha_1 k_1\a_1+\cdots+\alpha_p k_1\a_p
\end{equation}
and hence a multiple of $\b$ has to be in $\sA$. Thus, a necessary condition for  $\sA$ and $\langle \b\rangle$ to be glued is 
that a multiple of $\b$ belongs to $\sA$. 
We will see in theorem \ref{thmSimplesplit} that this condition is also sufficient.

\medskip
On the other hand, a sufficient condition that guarantees that $\sA$ and $\langle \b\rangle$ can be glued is that $\b\in\sA$. The following result is more precise and states that $\sA$ and $\langle \b\rangle$ can be glued in many different ways in that case:

\begin{proposition}\label{propGlueBinA}
Suppose that $\b\in \sA$. For all positive integers $k_1, k_2$ such that $k_1$ is relatively prime to $\gcd (b_1, b_2,\dots, b_n)$,
if $C=k_1A\cup \{k_2\b\}$, then $\sC$ is a gluing of $\sA$ and $\langle \b\rangle$.
\end{proposition}

\begin{proof}
As observed in the introduction, one can assume without loss of generality that $k_1$ and $k_2$ are relatively prime.
Since $\b\in \sA$, there exist integers $d_j \in \N$ such that 
$$b_i = \sum _{j=1}^p a_{ij}d_j\,,\ \forall i,\ 1\leq i\leq n\,.$$
So, $\rho = y^{k_1} -\prod_{j=1}^p x_j^{d_jk_2} \in I_C$ and we will show that $I_C = I_A+ \langle\rho\rangle$.  

\smallskip
First, let $w =y^{\gamma} \prod_j x_j^{\alpha_j}-\prod_j x_j^{\beta_j}$ 
be a arbitrary binomial in $I_C$ involving the variable $y$.  
Then we must have $\phi_C(w) = 0$, and hence 
$$\gamma k_2b_i + \sum_{j=1}^p \alpha_j k_1a_{ij} - \sum _{j=1}^p \beta_j k_1 a_{ij}= 0.$$
Since we assumed that $k_1$ and $k_2$ are relatively prime, 
$k_1| \gamma b_i$ for all $i$.  But by the fact that $k_1 $ is relatively prime to $\gcd(b_1, \ldots, b_n)$, we see that $k_1|\gamma$.  Let $\gamma = k_1r$. 

\smallskip
Now one has that $w =y^{k_1r} \prod_{j}x_j^{\alpha_j}-\prod_jx_j^{\beta_j}  \in I_C$.  
If we set $\rho^{[r]}=y^{k_1 r} -\prod_{j=1}^p x_j^{d_jk_2r}$, which is a multiple of $\rho$ and hence an element in $I_C$,
one has that  $w-\rho^{[r]}\prod_jx_i^{\alpha_j}  = \prod_jx_j^{d_jk_2r+\alpha_j}-\prod_jx_j^{\beta_j} \in I_C $ and it does not
involve the variable $y$ so it belongs to $I_A$.
So, $w\in I_A+\langle\rho\rangle$ and we have proved that $I_C = I_A+\langle\rho\rangle$. 
\end{proof}

\begin{remark}{\rm
Note that when one has a semigroup $\sA$ that is not minimally generated by $A$, for example if $\a_p\in \langle A'\rangle$
for $A'=A\setminus\{\a_p\}$, then we are in the situation described in proposition \ref{propGlueBinA} and $\sA$ is a gluing of $ \langle A'\rangle$
and $\langle\a_p\rangle$. This is somehow what we could call a trivial gluing.
}\end{remark}

Now observe that if one looks more precisely at condition (\ref{eq1}), one gets that if for $C=k_1 A\cup \{k_2\b\}$, one has that $\sC$ is a gluing
of $\sA$ and $\langle \b\rangle$, then $X=\begin{pmatrix} k_1\alpha_1\\ \vdots\\ k_1\alpha_p\end{pmatrix}$ is 
a solution to the system $A\cdot X=dk_2\b$ that belongs to $\N^p$.
This gives a necessary condition on $A$ and $\b$ for $\sA$ and $\langle\b\rangle$ to be glued: 
the system $A\cdot X=d\b$ must have a solution in $\N^p$
for some integer $d\geq 1$.

\medskip
This condition is also sufficient: if the system $A\cdot X=d\b$ has a solution in $\N^p$
for some integer $d\geq 1$, then $d\b\in\sA$ and, by proposition \ref{propGlueBinA}, for $C=k_1 A\cup \{d\b\}$, one has
that $\sC$ is a gluing of $\sA$ and $\langle \b\rangle$ for any integer $k_1$ relatively prime to $d$ and to $\gcd(b_1,\ldots,b_n)$.
This shows the following characterization:

\begin{theorem}\label{thmSimplesplit}
If $B$ has only one element, i.e., $B=\{\b\}$, then $\sA$ and $\sB$ can be glued \iff{} 
a multiple of $\b$ belongs to $\sA$, equivalently, \iff{}
the system $A\cdot X=d\b$ has a solution in $\N^p$ for some $d\in\N$.
\end{theorem}

\begin{example}{\rm
If $A=\{\begin{pmatrix} 1\\ 2\end{pmatrix}, \begin{pmatrix} 2\\ 1\end{pmatrix}\}$ and $B=\{\b\}$ with $\b=\begin{pmatrix} 3\\ 0\end{pmatrix}$,
it is clear that for all positive integer $d$, the system $A\cdot \begin{pmatrix} x\\ y\end{pmatrix}=d\b$ has no solution in $\N^2$ and hence
$\sA$ and $\sB$ can not be glued by theorem \ref{thmSimplesplit}.
}\end{example}

Assuming now that the conditions in theorem \ref{thmSimplesplit} are satisfied,
we can determine precisely the way to choose the 
integers $k_1$ and $k_2$ so that, for $C=k_1A\cup k_2B$, the semigroup $\sC$ is a gluing of $\sA$ and $\sB$.

\begin{notations}
{\rm
Assume that the system $A\cdot X=d\b$ has a solution in $\N^p$ for some $d\in\N$.
We will use the following notations:
\begin{itemize}
\item
$d(A,\b)$ is the smallest integer $d>0$ such that $A\cdot X=d\b$ has a solution in $\N^p$,
\item
$s(A,\b)$ is the smallest integer $s>0$ such that $A\cdot X=s\b$ has a solution in $\Z^p$.
\end{itemize}
}\end{notations}

\begin{theorem}\label{thmGlue1elt}
Let $A= \{\a_1,\ldots,\a_p\}\subset \mathbb N^n$ and $\b\in \N^n$ be such that a multiple of $\b$ is in $\sA$ and set $d=d(A,\b)$. 
Take two positive integers $k_1,k_2$ and set $C = k_1A\cup \{k_2\b\}$.
Then, $\sC$ is a gluing of $\sA$ and $\langle\b\rangle$ \iff{} ${d \over \gcd(d,k_2)}=s(A,k_2\b)$.
\end{theorem}

\begin{proof}
Set $\delta={d \over \gcd(d,k_2)}$. 
By definition of $d$, 
there exist $d_1,\ldots,d_p\in \mathbb N$ such that 
$$d\b = \sum_{j=1}^{p}  d_j\a_{j}\,.$$
Hence,
$k_1d(k_2 \b) = \sum_{j=1}^{p} k_2 d_j (k_1\a_{j})$ and, factoring out $\gcd(d,k_2)$, we get
$$
k_1\delta (k_2 \b) = \sum_{j=1}^{p} {k_2 \over \gcd(d,k_2)} d_j (k_1\a_{j})\,.
$$
Thus, $\rho = y^{k_1\delta} - \prod _{j=1}^{p} x_j^{{k_2\over \gcd(d,k_2)} d_j }\in I_C$.  

\medskip
Given an element $\alpha\in\Z^p$, we will denote 
$\alpha_+=\{j,\ 1\leq j\leq p\,/\ \alpha_j>0\}$ and $\alpha_-=\{j,\ 1\leq j\leq p\,/\ \alpha_j<0\}$.
Set $r=s(A,k_2\b)$. By definition,
$r$ is the smallest positive integer such that there exists a binomial of the form
$w=y^{k_1r} \prod_{j\in\alpha_+}x_j^{\alpha_j}-\prod_{j\in\alpha_-}x_j^{-\alpha_j}$ in $I_C$
for some $\alpha\in\Z^p$.
So if $w'=y^{k_1r'} \prod_{j\in\alpha'_+}x_j^{\alpha_j'}-\prod_{j\in\alpha'_-}x_j^{-\alpha_j'}$ is any binomial in $I_C$ of this form,
then $r\leq r'$ and one can write $r'=qr+r''$ for some $0\leq r''<r$. Setting
$w^{[q]}=y^{k_1rq} \prod_{j\in\alpha_+}x_j^{\alpha_j q}-\prod_{j\in\alpha_-}x_j^{-\alpha_j q}$, 
which is a multiple of $w$ and hence belongs to $I_C$, one has that
$\prod_{j\in\alpha_+}x_j^{\alpha_j q}w'-y^{k_1r''} \prod_{j\in\alpha'_+}x_j^{\alpha_j'}w^{[q]}= 
y^{k_1r''} \prod_{j\in\alpha'_+}x_j^{\alpha_j'} \prod_{j\in\alpha_-}x_j^{-\alpha_j q}-
\prod_{j\in\alpha'_-}x_j^{-\alpha_j'}\prod_{j\in\alpha_+}x_j^{\alpha_j q}\in I_C$.
Since $I_C$ is a prime binomial ideal, one can simplify this binomial by the common factor of the two monomials
and get a binomial of the form $y^{k_1r''} \prod_{j\in\beta_+}x_j^{\beta_j}-\prod_{j\in\beta_-}x_j^{-\beta_j}$ in $I_C$
for some $\beta\in\Z^p$.
By minimality of $r$, it implies that $r''=0$ and so $r$ divides $r'$. 
Applying this to $w'=\rho$, one gets that $r$ divides $\delta$ and, in particular, $r\leq\delta$.

\medskip
If $r<\delta$, then $\sC$ is not a gluing because in this case $w\in I_C$ and $w\notin I_A+\langle\rho\rangle$.

\medskip
If $r=\delta$, then $\rho=y^{k_1r} - \prod _{j=1}^{p} x_j^{{k_2\over \gcd(d,k_2)} d_j }$ and
if $w'$ is any other binomial in $I_C$ which is the difference between two monomials of disjoint supports, 
$w'=y^{k_1qr} \prod_{j\in\alpha'_+}x_j^{\alpha_j'}-\prod_{j\in\alpha'_-}x_j^{-\alpha_j'}$,
then $w'- \prod_{j\in\alpha_+}x_{j}^{\alpha_j'}\rho^{[q]}= 
\prod _{j=1}^{p} x_j^{{k_2\over \gcd(d,k_2)} d_j q}\prod _{j\in\alpha'_+} x_j^{\alpha_j'}
-\prod_{j\in\alpha'_-}x_j^{-\alpha_j'}$
that belongs to $I_A$ since it is in $I_C$ and the variable $y$ is not involved. This shows that, in that case,
$I_C=I_A+\langle\rho\rangle$ and hence $\sC$ is a gluing.
\end{proof}

\begin{remark}{\rm
We see by theorem \ref{thmSimplesplit} that $\sA$ and $\sB$ can be glued \iff{} $A\cdot X= dB$ has a solution in the positive integers for some $d\in\N$.    Let $d$ be the smallest such positive integer.
Applying proposition \ref{propGlueBinA}, one gets that if one chooses $k_1$ and $k_2$ such that $k_1$ is anything that is relatively prime to $\gcd(b_1,\ldots,b_n)$ and $d$, and $k_2$ is any multiple of $d$, then for $C = k_1A\cup \{k_2\b\}$, $\sC$ is a gluing of $\sA$ and $\sB$.
But theorem \ref{thmGlue1elt} is more precise and specifies exactly how to pick integers $k_1, k_2$ so that for $C = k_1A\cup \{k_2\b\}$, $\sC$ is a gluing of $\sA$ and $\sB$.
When $k_1$ is anything that is relatively prime to $\gcd(b_1,\ldots,b_n)$ and $d$ and $k_2$ is any multiple of $d$, then 
$\delta={d \over \gcd(d,k_2)}=1$ and $r=s(A,k_2\b)=1$ so obviously $\delta=r$ but this could also happen when the previous
condition does not hold as examples \ref{exGluing1eltCM} shows.
}
\end{remark}

\begin{corollary}\label{corSqfree}
Let $A= \{\a_1,\ldots,\a_p\}\subset \mathbb N^n$ and $\b\in \N^n$ be such that a multiple of $\b$ is in $\sA$,
set $d=d(A,\b)$ and $s=s(A,\b)$, and assume that $d$ is squarefree. 
Then $d= st$ for some $t\in\N$, and for any positive integers $k_1,k_2$ such that $k_2$ is a multiple of $t$, 
one has that  for $C = k_1A\cup \{k_2\b\}$, $\sC$ is a gluing of $\sA$ and $\langle\b\rangle$.
\end{corollary}

\begin{proof}
Set $s=s(A,\b)$.  Now, since $s$ is the smallest positive integer such that $A\cdot X =s\b$ has a solution in integers, and $A\cdot X =d \b$ has a solution in positive integers and hence in integers, we must have $s|d$. Thus, $d= st$ for some $t\in \mathbb N$. 
Since $d$ is square free, $\gcd(s,t) =1$.  On the other hand, $k_2 = tu$ by hypothesis.  
So, $ \gcd(k_2, d) = \gcd (tu, ts) = tu_1$  where   $u_1=\gcd(u,s)$.
Then $u=u_1u'$ with $u'$ relatively prime to $s$, because $d$ and hence $s$ is squarefree. 
Further, $\delta={d \over \gcd(d,k_2)}={s\over u_1}$, so $d = st = tu_1\delta$.
Now, $d$ is squarefree implies $\delta$ is relatively prime to $t u_1$. 
Setting $r=s(A,k_2\b)$, $r$ is the smallest positive integer such that   $A\cdot X = rk_2\b$ has a solution in integers.  
So,  $s|rk_2 = rtu_1u'$ and since $\gcd(s,tu')=1$, $s$ divides $ru_1$. But $s=\delta u_1$ so $\delta$ divides $r$.
We saw in the proof of theorem \ref{thmGlue1elt} that $r$ divides $\delta$ so $\delta=r$,
and hence $\sC$ is a gluing of $\sA$ and $\langle\b\rangle$. 
\end{proof}

\begin{corollary}\label{corCMchar1elt}
Suppose that $\sA$ and $\langle \b\rangle$ can be glued, and let $\sC$ be a gluing of $\sA$ and $\langle\b\rangle$. 
Then, $\dim{\rA}=\dim{\rC}$ and $\depth (\rA)=\depth (\rC)$.
In particular, $\rC$ is Cohen-Macaulay \iff{} $\rA$ is Cohen-Macaulay.
\end{corollary}

\begin{proof}
Since  $\sA$ and $\langle\b\rangle$ can be glued,  by theorem \ref{thmSimplesplit}, we know that a multiple of $\b$ is in $\sA$ and it follows that  $\rk{A}=\rk{C}$,
so $\dim{\rA}=\dim{\rC}$. On the other hand, $\pd(\rC)=\pd(\rA)+1$ by theorem \ref{thmJpaa} (\ref{projdim}) and as the polynomial ring $R$
has one variable more than $R_A$, the Auslander-Buchsbaum formula shows that $\depth (\rA)=\depth (\rC)$.  Therefore, we see that, $\rC$ is Cohen-Macaulay \iff{} $\rA$ is Cohen-Macaulay.
\end{proof}

We end this section with a series of illustrating examples.

\begin{example}{\rm
Consider  
$
A=\{
\begin{pmatrix} 7\\0\end{pmatrix},\begin{pmatrix} 6\\2\end{pmatrix},\begin{pmatrix} 3\\8\end{pmatrix},\begin{pmatrix} 0\\9\end{pmatrix}\}
\subset\N^2
$ 
and $\b=\begin{pmatrix} 3\\4\end{pmatrix}$. 
For $C=A\cup\{\b\}$, one has that $\sC$ is a not gluing of $\sA$ and $\langle\b\rangle$ because using Singular \cite{Sing}, one can check that
$I_C=I_A+\langle\rho\rangle+\langle x_2x_4^2y-x_3^3, x_1^3x_4^2y-x_2^3x_3^2,
x_3^2y^5-x_1^3x_4^4\rangle$ for $\rho=y^6-x_2^3x_4^2$.
In this case $d(A,\b)=6>s(A,\b)=1$. 
Of course since $6\b\in\sA$, $\sA$ and $\langle\b\rangle$ can be glued by proposition \ref{propGlueBinA},
and for $C=A\cup\{6\b\}$, $\sC$ is a gluing of $\sA$ and $\langle\b\rangle$.
In this example, $6$ is the smallest positive integer such that this occurs. 
}\end{example}

\begin{example}\label{exGluing1eltCM}{\rm
For 
$A=\{
\begin{pmatrix} 5\\0\end{pmatrix},\begin{pmatrix} 3\\2\end{pmatrix},\begin{pmatrix} 2\\3\end{pmatrix},\begin{pmatrix} 0\\5\end{pmatrix}\}\subset\N^2
$
and $\b=\begin{pmatrix} 1\\1\end{pmatrix}$, it is clear that $5\b\in\sA$ so $\sA$ and $\sB$ can be glued by theorem \ref{thmSimplesplit}.
Using Singular \cite{Sing}, one can check that one can choose $k_1=k_2=1$, and for $C=A\cup B$, $\sC$ is a gluing of $\sA$ and $\sB$.
This is also given by theorem \ref{thmGlue1elt} since $d(A,\b)=s(A,\b)=5$. This example shows that
it is not necessary to choose $k_2$ as a multiple of $d(A,\b)$.
In this example, both $\rA$ and $\rC$ are Cohen-Macaulay of type 2.
}\end{example}

\begin{example}\label{exGluing1eltNotCM}{\rm
An example very similar to the previous one is given by 
$$A=\{
\begin{pmatrix} 5\\0\end{pmatrix},\begin{pmatrix} 4\\1\end{pmatrix},\begin{pmatrix} 1\\4\end{pmatrix},\begin{pmatrix} 0\\5\end{pmatrix}\}\subset\N^2
\quad\hbox{\rm and}\quad
\b=\begin{pmatrix} 1\\1\end{pmatrix}\,.$$
Again, one can choose  $k_1=k_2=1$ and for $C=A\cup B$, $\sC$ is a gluing of $\sA$ and $\sB$.
Here, $\rA$ and $\rC$ are not Cohen-Macaulay and their minimal free resolutions show as:
$$
0\rar R_A^2\rar R_A^6\rar R_A^5\rar \rA\rar 0 \quad\hbox{\rm and}\quad
0\rar R^2\rar R^8\rar R^{11}\rar R^6\rar \rC\rar 0\,.
$$
}\end{example}

\section{When the elements lie on a line}\label{secLine}

Consider two subsets $A=\{\a_1,\ldots,\a_p\}$ and $B=\{\b_1,\ldots,\b_q\}$ in $\mathbb N^n$ such that
all the elements in $B$ lie on a same line, i.e.,  $q\geq 2$ and
$B = \b \begin{bmatrix}u_1 & \ldots & u_q\end{bmatrix}$ 
for some $u_1,\ldots,u_q\in\N$ and 
$\b=\begin{pmatrix} b_1\\ \vdots\\ b_n\end{pmatrix}\in\N^n$. 
One can always assume, without loss of generality, that $\gcd(b_1,\ldots,b_n)=1$, $\gcd(u_1,\ldots,u_q)=1$ and
$\gcd(a_{ij}\,,1\leq i\leq n,\,1\leq j\leq p\,)=1$ if needed. 
If this does not occur, one can factor out the gcd in $A$ or $B$ and simplify to get another 
semigroup with the same semigroup ring.

\begin{proposition}\label{propLine}
Suppose that $\b\in \sA$ and  $B = \b \begin{bmatrix}u_1 & \ldots & u_q\end{bmatrix}$ for $u_1, \ldots, u_q \in \N$.
Let $C=k_1A\cup k_2B$, for positive integers $k_1, k_2$ such that $k_1\in\langle u_1, \ldots, u_q\rangle$.  Then  $\sC$ is a gluing of $\sA$ and $\sB$.
\end{proposition}

Note that, as we saw in the introduction, one can assume without loss of generality that $k_1$ and $k_2$ are relatively prime.
Before we can prove this proposition, we will show an easy preliminary result.
Let's start with some notation. As in the proof of theorem \ref{thmGlue1elt}, 
given an element $\alpha\in\Z^r$, denote 
$\alpha_+=\{j,\ 1\leq j\leq r\,/\ \alpha_j>0\}$ and $\alpha_-=\{j,\ 1\leq j\leq r\,/\ \alpha_j<0\}$.
Then, a binomial in $R=k[x_1,\ldots,x_p,y_1,\ldots,y_q]$ which is the difference between two monomials of disjoint supports
is always of the form
$$
\prod_{j\in \alpha_+}y_j^{\alpha_j} \prod_{j\in \beta_-} x_j^{-\beta_j}-
           \prod_{j\in \alpha_-}y_j^{-\alpha_j}\prod_{j\in \beta_+}x_j^{\beta_j}
$$
for some $\alpha\in\Z^q$ and $\beta\in\Z^p$.

\begin{lemma}\label{lemIntegerS}
Given a binomial 
$w = \prod_{j\in \alpha_+}y_j^{\alpha_j} \prod_{j\in \beta_-} x_j^{-\beta_j}-
           \prod_{j\in \alpha_-}y_j^{-\alpha_j}\prod_{j\in \beta_+}x_j^{\beta_j}$
in $R$, one has that
$w\in I_C$ \iff{} there exists $s\in\Z$ such that
$$
\sum_{j=1}^{q} \alpha_j u_j=k_1s
\quad\hbox{and}\quad 
A\cdot \beta=k_2s\b\,.
$$
\end{lemma}

\begin{proof}
The binomial $w$ is in $I_C$ \iff{} 
$\sum_{j=1}^{q} \alpha_j (k_2\b_j)=\sum_{j=1}^{p} \beta_j(k_1\a_j)$, i.e.,
\begin{equation}\label{eq2}
k_2b_i\sum_{j=1}^{q} \alpha_j u_j=k_1\sum_{j=1}^{p} \beta_ja_{ij},
\ \forall i=1,\ldots,n\,.
\end{equation}
Since we have assumed that $\gcd(k_1,k_2)=1$ and $\gcd(b_1,\ldots,b_n)=1$,  we deduce that $k_1$ has to divide $\sum_{j=1}^{q} \alpha_j u_j$:
there exists $s\in\Z$ such that $\sum_{j=1}^{q} \alpha_j u_j=k_1s$. Going back to (\ref{eq2}), one gets that 
$\sum_{j=1}^{p} \beta_ja_{ij}=k_2b_is$ for all $i$, i.e., $A\cdot \beta=k_2s\b$.

\medskip
Conversely, if  $\sum_{j=1}^{q} \alpha_j u_j=k_1s$ and $A\cdot \beta=k_2s\b$ for some $s\in\Z$, then
$\sum_{j=1}^{q} \alpha_j (k_2\b_j)= k_2(\sum_{j=1}^{q} \alpha_ju_j) \b= k_1k_2s\b$ and
$\sum_{j=1}^{p} \beta_j(k_1\a_j)= k_1\sum_{j=1}^{p} \beta_j\a_j=k_1A\cdot \beta=k_1k_2s\b$
and hence $w\in I_C$.
\end{proof}

\begin{definition}\label{defIntegerS}
{\rm
Given a binomial $w\in I_C$ as in lemma \ref{lemIntegerS}, we call the integer $s$ the {\it level} 
of $w$ and denote it by $s(w)$:
$s(w)=\frac{\sum_{j=1}^{q} \alpha_j u_j}{k_1}$.
}\end{definition}

On the other hand, since we have assumed in proposition \ref{propLine} that $\b\in\sA$ and $k_1\in\sS$, one has that:
\begin{itemize}
\item
$\b=\sum_{j=1}^{p}d_j\a_j$ for some $d_1,\ldots,d_q\in\N$;
\item
$k_1=\sum_{j=1}^{q}v_ju_j$ for some $v_1,\ldots,v_q\in\N$.
\end{itemize}

Then, $k_1k_2\b=
\sum_{j=1}^{q}v_j(k_2\b_j)=
\sum_{j=1}^{p}k_2d_j(k_1\a_j)
$, and hence, for $C=k_1 A\cup k_2 B$, one has that
the binomial
\begin{equation}\label{eqRho}
\rho=\prod_{j=1}^{q}y_j^{v_j}-\prod_{j=1}^{p}x_j^{d_j k_2}
\end{equation}
belongs to $I_C$.

\begin{remark}{\rm
The binomial $\rho$ in (\ref{eqRho}) has level 1 because 
$\sum_{j=1}^{q}v_ju_j=k_1$. 
Moreover, $A\cdot \begin{pmatrix} k_2d_1\\ \vdots\\ k_2d_p\end{pmatrix}=k_2\b$
by lemma \ref{lemIntegerS}.
}\end{remark}

We are now ready to prove proposition \ref{propLine}.

\begin{proof}[Proof of proposition \ref{propLine}]
We will show that $I_C\subset I_A+I_B+(\rho)$ where $\rho$ is the binomial defined in (\ref{eqRho}) since we already know that
the reverse inclusion holds. 

\medskip
Let $w = \prod_{j\in \alpha_+}y_j^{\alpha_j} \prod_{j\in \beta_-} x_j^{-\beta_j}-
           \prod_{j\in \alpha_-}y_j^{-\alpha_j}\prod_{j\in \beta_+}x_j^{\beta_j}$
be a binomial in $I_C$.
By lemma \ref{lemIntegerS}, one has that $\sum_{j=1}^{q} \alpha_j u_j=k_1s(w)$
and $A\cdot \beta=k_2s(w)\b$. 
Now, consider the binomial
$\theta= \prod_{j}y^{s(w)v_j} \prod_{j\in \beta_-} x_j^{-\beta_j}-\prod_{j\in \beta_+}x_j^{\beta_j}$. 
One has that $\theta  \in I_C$ by lemma \ref{lemIntegerS}
because 
$\sum_{j=1}^{q}s(w)v_ju_j=k_1s(w)$
and $A\cdot\beta= s(w)k_2\b$. 

\medskip
For $\rho ^{[s(w)]}=\prod_{j=1}^{q}y_j^{s(w)v_j}-\prod_{j=1}^{p}x_j^{s(w)d_j k_2}$, that belongs to $I_C$ because it is a multiple of $\rho$,
one has that 
$\theta - \rho ^{[s(w)]}\prod_{j\in \beta_-}x_j^{-\beta_j} = 
\prod_{j=1}^px_j^{s(w)d_jk_2}\prod_{j\in \beta_-}x_j^{-\beta_j} -\prod_{j\in \beta_+}x_j^{\beta_j} \in I_A$.  
This shows that $\theta \in I_A+(\rho)$.

\medskip
Now, $w- \prod_{j\in \alpha_-}y_j^{-\alpha_j} \theta = 
\prod_{j\in \alpha_+}y^{\alpha_j} \prod_{j\in \beta_-} x_j^{-\beta_j}-
\prod_{j\in \alpha_-}y_j^{-\alpha_j} \prod_{j}y^{s(w)v_j} \prod_{j\in \beta_-} x_j^{-\beta_j}=
 \prod_{j\in \beta_-} x_j^{-\beta_j}
 (\prod_{j\in \alpha_+}y^{\alpha_j}-\prod_{j\in \alpha_-}y_j^{-\alpha_j} \prod_{j}y^{s(w)v_j})\in I_B$,
and hence $w- \prod_{j\in \alpha_-}y_j^{-\alpha_j} \theta \in I_B$.
This shows that $w\in I_A+I_B+(\rho)$ and we are done.
\end{proof}

Note that in proposition \ref{propLine}, there is no condition on $k_2$ and one can use $k_2$ if we have that $\b\notin\sA$ but $d\b\in\sA$ for some integer $d>1$. 
This shows that if a multiple of $\b$ belongs to $\sA$, then $\sA$ and $\sB$ can be glued.
As in the case of a simple split, this is indeed a characterization.

\begin{theorem}\label{thmLineChar}
If  the elements in $B$ lie on a line, i.e., $q\geq 2$ and there exists $\b\in\N^n$ such that $B=\b [u_1\ldots u_q]$
for some $u_1,\ldots,u_q\in\N$, 
then $\sA$ and $\sB$ can be glued \iff{} 
a multiple of $\b$ belongs to $\sA$, equivalently, \iff{}
the system $A\cdot X=d\b$ has a solution in $\N^p$ for some $d\in\N$.
\end{theorem}

\begin{proof}
If a multiple of $\b$ belongs to $\sA$, say $d\b\in\sA$, then by chosing $k_2=d$ in proposition \ref{propLine}, one gets that $\sA$ and $\sB$ can be glued.
Conversely, if there exists positive integers $k_1, k_2$ such that for $C=k_1A\cup k_2B$, $\sC$ is a gluing of $\sA$ and $\sB$, then
one has a binomial $\rho\in I_C$ of the form $\rho=y_1^{\beta_1}\cdots y_q^{\beta_q}-x_1^{\alpha_1}\cdots x_p^{\alpha_p}$. Thus,
$$
\beta_1k_2\b_1+\cdots+\beta_qk_2\b_q=\alpha_1k_1\a_1+\cdots+\alpha_pk_1\a_p
$$
and since $\b_j=u_j\b$ for all $j=1,\ldots,q$, one gets that $k_2(\sum_{j=1}^{q}\beta_j u_j)\b\in\sA$.
\end{proof}

Using essentially the same argument as in corollary \ref{corCMchar1elt}, one can easily show the following:

\begin{corollary}\label{corCMcharLine}
Suppose that the elements in $B$ lie on a line and that $\sA$ and $\sB$ can be glued, and consider $\sC$, a gluing of $\sA$ and $\sB$. 
Then, $\dim{\rA}=\dim{\rC}$ and $\depth (\rA)=\depth (\rC)$.
In particular, $\rC$ is Cohen-Macaulay \iff{} $\rA$ is Cohen-Macaulay.
\end{corollary}

\begin{remark}{\rm
Note that in corollaries \ref{corCMchar1elt} and \ref{corCMcharLine}, the dimension and depth of $\rB$ are not involved.
The reason is that in the first case, $I_B=(0)\subset k[y]$, while in the second, $\sB$ behaves like the
associated numerical semigroup $\sD$. So  in both cases, $\dim{\rB}=\depth (\rB)=1$.
In particular, $\rB$ is always Cohen-Macaulay.
}\end{remark}

\begin{example}\label{exLineCM}
{\rm
Consider 
$
A=\{
\begin{pmatrix} 5\\0\end{pmatrix},\begin{pmatrix} 3\\2\end{pmatrix},\begin{pmatrix} 2\\3\end{pmatrix},\begin{pmatrix} 0\\5\end{pmatrix}\}
\subset\N^2
$
and $B=\b 
\begin{bmatrix}
11 & 17 & 25 & 19
\end{bmatrix}$
for $\b=\begin{pmatrix} 1\\1\end{pmatrix}$. Since $5\b\in\sA$, one has that $\sA$ and $\sB$ can be glued by theorem \ref{thmLineChar}.
Moreover, proposition \ref{propLine} tells us how to do it. Setting $S=\{11,17,25,19\}\subset\N$ and choosing, for example,
$k_1=28\in\sS$ and $k_2=5$, 
one has that for
$$
C=k_1A\cup k_2B=\{
\begin{pmatrix} 140\\0\end{pmatrix},\begin{pmatrix} 84\\56\end{pmatrix},\begin{pmatrix} 56\\84\end{pmatrix},\begin{pmatrix} 0\\140\end{pmatrix},
\begin{pmatrix} 55\\55\end{pmatrix},\begin{pmatrix} 85\\85\end{pmatrix},\begin{pmatrix} 125\\125\end{pmatrix},\begin{pmatrix} 95\\95\end{pmatrix}
\}\,,
$$
$\sC$ is a gluing of $\sA$ and $\sB$, and $I_C=I_A+I_B+\langle\rho\rangle$ for $\rho=x_1x_4-y_1y_2$.
The resolutions of $\rA$ and $\rB$ look as follows:
$$
0\rar R_A^2\rar R_A^3\rar \rA\rar 0 \quad\hbox{\rm and}\quad
0\rar R_B\rar R_B^5\rar R_B^5\rar \rB\rar 0
$$
and by theorem \ref{thmJpaa} (\ref{betti}), the resolution of $\rC$ is:
$$
0\rar R^2\rar R^{15}\rar R^{39}\rar R^{48}\rar R^{30}\rar R^9\rar \rC\rar 0\,.
$$
In this example, both $\rA$ and $\rC$ are Cohen-Macaulay ($\rB$ is also Cohen-Macaulay and degenerate as already observed). 
}\end{example}

\begin{example}\label{exLineNotCM}
{\rm
In example \ref{exLineCM}, if we substitute $A$ for 
$
A=\{
\begin{pmatrix} 5\\0\end{pmatrix},\begin{pmatrix} 4\\1\end{pmatrix},\begin{pmatrix} 1\\4\end{pmatrix},\begin{pmatrix} 0\\5\end{pmatrix}\}
\subset\N^2
$
and keep the rest of the data, we get a gluing of $\sA$ and $\sB$. The difference is that in this case neither $\rA$ nor $\rC$ are Cohen-Macaulay.
}\end{example}

\begin{example}\label{exLineIterated1elt}
{\rm
In example \ref{exLineCM}, if we change $k_1$ and take $k_1=2$ which is not in $\sS$, it is not a gluing. But in that case, 
it behaves like an iteration of simple splits, the situation discribed in section \ref{secSimplesplit}:
$I_C=I_A+\langle\rho_1\rangle+\langle\rho_2\rangle+\langle\rho_3\rangle+\langle\rho_4\rangle$ for
$\rho_1=y_1^2-x_1^{11}x_4^{11}$, $\rho_2=y_2-y_1x_1^3x_4^3$, $\rho_3=y_3-y_2x_1^4x_4^4$ and $\rho_4= y_4-y_2x_1x_4$.
The ideal $I_B=I_D$ is not involved here (even if it is, of course, contained in $I_C$) and $I_C$ does not coincide with $I_A+I_B+\langle\rho_1\rangle$.
}\end{example}

\begin{example}{\rm
If we now take $k_1=26$ in the same example, then $I_C$ is minimally generated by 13 binomials. One of them is $\rho=y_1^3y_4-x_1^2x_4^2$.
In this case, if $C=k_1A\cup k_2B$, $\sC$ is not a gluing ($I_C\neq I_A+I_B+\langle\rho\rangle$) and it is not either an iteration of simple splits.
}\end{example}

\section{A direct consequence: the case $n=2$}\label{secConsequences}

Putting all together, we can answer completely to question \ref{weakerquest} when $n=2$:

\begin{theorem}
Let $A=\{\a_1,\ldots,\a_p\}$ and $B=\{\b_1,\ldots,\b_q\}$ be two finite subsets of $\N^2$ such that $\rA$ and $\rB$ are Cohen-Macaulay. 
Then, $\sA$ and $\sB$ can be glued \iff{}
one of the two subsets, for example $B$, satisfies one of the following conditions:
\begin{itemize}
\item
either $B$ has one single element $\b$, i.e., $q=1$,
\item
or $q\geq 2$ and $B = \b \begin{bmatrix}u_1 & \ldots & u_q\end{bmatrix}$ for
some $u_1,\ldots,u_q\in\N$,
\end{itemize}
and  the system $A\cdot x=d\b$ has a solution in $\N^p$ for some $d\in\N$.

In this case,
if $\sC$ is a gluing of $\sA$ and $\sB$, then $\rC$ is Cohen-Macaulay.
\end{theorem}

\begin{proof}
By theorem \ref{thmNondegen}, if $\sA$ and $\sB$ can be glued, then one of the two semigroups, for example $\sB$, has to be degenerate. 
Since $n=2$, $\sB$ is degenerate if and only if $\rk{B}=1$, i.e., either
$B=\{\b\}$ (simple split) or $q\geq 2$ and $B = \b \begin{bmatrix}u_1 & \ldots & u_q\end{bmatrix}$ for $u_1,\ldots,u_q\in\N$
(the elements in $B$ lie on a line). In both cases, $\sA$ and $\sB$ can be glued \iff{} a multiple of the corresponding vector $\b$ belongs to $\sA$ 
by theorems \ref{thmSimplesplit} and \ref{thmLineChar}, and the result follows.
Finally, it follows from corollaries \ref{corCMchar1elt} and \ref{corCMcharLine} that $\rC$ is Cohen-Macaulay when $\sC$ is a gluing of $\sA$ and $\sB$.  
\end{proof}

\end{document}